\documentclass[11pt]{amsart} 
\usepackage{amssymb,amsmath,latexsym,enumerate,graphicx,bbm,mathptmx,microtype,cite}

\newtheorem{theorem}{Theorem} 
\newtheorem{corollary}[theorem]{Corollary}

\newtheorem{exam}{Example}

\newtheorem*{rem}{Remarks}

\newcommand\commentout[1]{}

\newcommand\ZZ{\mathbb{Z}}

\newcommand\bm{\mathbf{m}}

\begin{document}

\title{Frobenius Coin-Exchange Generating Functions}

\author{Leonardo Bardomero}
\author{Matthias Beck}
\address{Department of Mathematics\\
         San Francisco State University\\
         San Francisco, CA 94132\\
         U.S.A.}
\email{[josebardomero,becksfsu]@gmail.com}


\begin{abstract}
We study variants of the \emph{Frobenius coin-exchange problem}: Given $n$ positive relatively prime parameters, what is the largest integer that cannot be represented as a nonnegative integral linear combination of the given integers? 
This problem and its siblings can be understood through generating functions with 0/1 coefficients according to whether or not an integer is representable. In the 2-parameter
case, this generating function has an elegant closed form, from which many corollaries follow, including a formula for the Frobenius problem.
We establish a similar closed form for the generating function indicating all integers with exactly $k$ representations, with similar wide-ranging corollaries.
\end{abstract}


\thanks{We thank Federico Ardila, Yitwah Cheung, two anonymous referees, and the \textsc{Monthly} editors for insightful comments about this work.}


\date{13 May 2019, to appear in the \emph{American Mathematical Monthly}}

\maketitle


\section{The Story}

Imagine we replace the penny in the US currency coins by a 7-cent coin.  One might argue that the resulting new coin system is a
bit less practical than the old, but it is also more (mathematically) interesting: now there are some cent amounts (such as 3 and
8 cents) that cannot be made up using our coins. On the other hand, it is a charming exercise---because 5 and 7 happen to be
relatively prime---that every sufficiently large amount of money \emph{can} be changed; in fact, there are twelve cent amounts that
cannot be made up with 5- and 7-cent coins, the largest being 23 cents.
(The simple fact that 5 and 7 are relatively prime is crucial---if the greatest common divisor of our coin denominations were $d$, we
could not change any amount that is not a multiple of~$d$.)

Naturally, nothing stops us (mathematicians) from generalizing this setting, and so for fixed
positive relatively prime integers $a_1, a_2, \dots, a_n$, (that is, $\gcd(a_1, a_2, \dots, a_n) = 1$),
we say that a nonnegative integer $x$ is \emph{$\left( a_1, a_2, \dots, a_n \right)$-representable} if 
\begin{equation}\label{repequ}
  x = m_1 a_1 + m_2 a_2 + \dots + m_n a_n
\end{equation} 
for some $m_1, m_2, \dots, m_n \in \ZZ_{ \ge 0 }$. 
Let $R_0 \left( a_1, a_2, \dots, a_n \right)$ be the set of all positive integers that are not $\left( a_1, a_2, \dots, a_n \right)$-representable.
Because $a_1, a_2, \dots, a_n$ are relatively prime, $R_0 \left( a_1, a_2, \dots, a_n 
\right)$ is finite, and so three natural questions about this set are:
\begin{itemize}
  \item What is the largest number $g_0 \left( a_1, a_2, \dots, a_n \right)$ in $R_0 \left( a_1, a_2, \dots, a_n 
\right)$?
  \item What is the cardinality $c_0 \left( a_1, a_2, \dots, a_n \right)$ of $R_0 \left( a_1, a_2, \dots, a_n 
\right)$?
  \item What is the sum $s_0 \left( a_1, a_2, \dots, a_n \right)$ of all elements in $R_0 \left( a_1, a_2, \dots, a_n 
\right)$?
\end{itemize}
The first question is the \emph{linear Diophantine problem of Frobenius} (it has many alternative names, such as the
\emph{coin-exchange problem} and the \emph{chicken nuggets problem}), and its solution $g_0 \left( a_1, a_2, \dots, a_n \right)$ is called the \emph{Frobenius number} of the
parameter set $\{ a_1, a_2, \dots, a_n \}$.
One of the appealing aspects of the Frobenius problem and its variants is that they can be easily explained. 
There are many reasons to be interested in the set $R_0 \left( a_1, a_2, \dots, a_n 
\right)$ for fixed $a_1, a_2, \dots, a_n$; the mathematical basis is the
semigroup $S_0 \left( a_1, a_2, \dots, a_n \right)$ generated by $a_1, a_2, \dots, a_n$, and then $R_0 \left( a_1, a_2, \dots, a_n \right) = \ZZ_{ \ge 0 } \setminus S_0 \left( a_1, a_2, \dots, a_n \right)$. 
For details about the Frobenius problem, including numerous applications, we recommend 
two classic \textsc{Monthly} articles~\cite{nijenhuis,wilfmonthly} and the monograph~\cite{ramirezbook}.

Our three questions about $R_0 \left( a_1, a_2, \dots, a_n 
\right)$ are, in general, wide open, but they have strikingly simple answers for $n=2$:
\begin{itemize}
  \item $g_0 (a, b) = (a - 1) (b - 1) - 1$;
\vspace{.1in}
  \item $c_0 (a, b) = \frac 1 2 (a - 1) (b - 1)$;
\vspace{.1in}
  \item $s_0 (a, b) = \frac{ 1 }{ 12 } (a - 1) (b - 1) ( 2 a b - a - b - 1)$.
\end{itemize}
The first two formulas go back to at least Sylvester; his paper \cite{sylvester} gives both $c_0(a,b)$ and a clear indication
that he knew $g_0(a,b)$. The third formula is much younger and seems to have first been proved by Brown--Shiue~\cite{brownshiue}.
One can derive all three formulas at once from the following generating function identity.

\begin{theorem}\label{thm:Sgenfct}
Given relatively prime positive integers $a$ and $b$, let
$S_0 (a,b) = \{ ma+nb : \, m, n \in \ZZ_{ \ge 0 } \}$. Then
\[
  \sum_{ j \in S_0 (a,b) } z^j \ = \ \frac{ 1 - z^{ ab } }{ (1-z^a) (1-z^b) } \, .
\]
\end{theorem}

Theorem~\ref{thm:Sgenfct} seems to have first been proved by Sz\'ekely--Wormald~\cite{szekelywormald} and independently by Sert\"oz--\"Ozl\"uk~\cite{sertozozluk2}; 
its usefulness to our three original questions were noticed already in the aforementioned~\cite{brownshiue}: namely, we observe that
\[
  p_0 \left( a, b; z \right)
  \ := \ \sum_{ j \in R_0 \left( a, b \right) } z^j
  \  = \ \frac{ 1 }{ 1-z } - \frac{ 1 - z^{ ab } }{ (1-z^a) (1-z^b) }
\]
is a polynomial disguised as a rational function, and since
\begin{itemize}
\vspace{.04in}
  \item $g_0 (a, b)$ equals the degree of $p_0 \left( a, b; z \right)$,
\vspace{.04in}
  \item $\displaystyle c_0 (a, b) = \lim_{ z \to 1 } p_0 \left( a, b; z \right)$, and
\vspace{.04in}
  \item $\displaystyle s_0 (a, b) = \lim_{ z \to 1 } p_0' \left( a, b; z \right)$,
\end{itemize}
the formulas stated above can be computed by a (patient) calculus student.
Theorem~\ref{thm:Sgenfct} is at the heart of this article, and in the interest of self-containment, we will give a proof below.
It is a curious fact---and one that is the subject of the \textsc{Monthly} papers
\cite{carlitzmonthly,moreemonthly}---that we have the alternative form
\[
  \sum_{ j \in S_0 (a,b) } z^j \ = \ \frac{ \Phi_{ ab } (z) }{ 1-z } \, ,
\]
where $\Phi_n(z)$ denotes the $n$th cyclotomic polynomial.

Our goal is to extend the machinery provided by Theorem~\ref{thm:Sgenfct} and its
consequences to a recent variant of the Frobenius problem that has attracted some attention
in the research community. Namely, we consider the set $R_k \left( a_1, a_2, \dots, a_n \right)$ consisting of all integers with
exactly $k$ representations in the form~\eqref{repequ}, and ask for
\begin{itemize}
  \item the largest number $g_k \left( a_1, a_2, \dots, a_n \right)$ in $R_k \left( a_1, a_2, \dots, a_n \right)$,
  \item the cardinality $c_k \left( a_1, a_2, \dots, a_n \right)$ of $R_k \left( a_1, a_2, \dots, a_n \right)$, and
  \item the sum $s_k \left( a_1, a_2, \dots, a_n \right)$ of all elements in $R_k \left( a_1, a_2, \dots, a_n \right)$.
\end{itemize}
These are, naturally, hard questions, but there are again answers for $n=2$, both proved
in~\cite{frobnote}:\footnote{
The formula for $c_k (a, b)$ appears differently in~\cite{frobnote}; the difference stems from considering positive vs.\ nonnegative
integers.
}
\begin{itemize}
  \item $g_k (a, b) = (k+1)ab - a - b$ 
  \item $c_k (a, b) = ab$ for $k \ge 1$.
\end{itemize}
Our main contribution is the following generalization of Theorem~\ref{thm:Sgenfct}, which
will, among other things, allow us to add the missing third bulleted item to the above list.

\begin{theorem}\label{thm:kgenfct}
Given relatively prime positive integers $a$ and $b$, let $S_k (a,b)$ consist of all integers with
more than $k$ representations in the form $ma+nb$ with $m, n \in \ZZ_{ \ge 0 }$. Then
\[
  \sum_{ j \in S_k (a,b) } z^j \ = \ \frac{ z^{ abk } (1 - z^{ ab }) }{ (1-z^a) (1-z^b) } \, .
\]
Consequently, for $k \ge 1$, the polynomial indicating all integers with exactly $k$
representations is
\[
  p_k \left( a, b; z \right)
  \ := \ \sum_{ j \in R_k (a,b) } z^j
  \  = \ \frac{ z^{ ab(k-1) } (1 - z^{ ab })^2 }{ (1-z^a) (1-z^b) } \, .
\]
\end{theorem}

Naturally, this theorem gives an alternative proof for the above formulas for $g_k(a,b)$ (by
computing the degree of $p_k \left( a, b; z \right)$) and $c_k(a,b)$ (by computing $\displaystyle \lim_{ z \to 1 } p_k \left( a, b; z \right)$), and because $\displaystyle s_k (a, b) = \lim_{ z \to 1 } p_k' \left( a, b; z \right)$, Theorem~\ref{thm:kgenfct} yields:

\begin{corollary}
Let $a$ and $b$ be relatively prime positive integers and $k \ge 1$. Then
$
  s_k (a, b) \ = \ \tfrac 1 2 \, ab \left( 2abk - a - b \right) .
$
\end{corollary}

But Theorem~\ref{thm:kgenfct} reveals more, namely, that the integers in $R_k (a,b)$ (for $k \ge
1$) are aligned in a highly structured way, as we may write
\begin{equation}\label{eq:polynomialversion}
  p_k \left( a, b; z \right)
  \ = \ \sum_{ j \in R_k (a,b) } z^j
  \ = \ z^{ ab(k-1) } \left( 1 + z^a + z^{ 2a } + \dots + z^{ (b-1)a } \right) \left( 1 + z^b
+ z^{ 2b } + \dots + z^{ (a-1)b } \right) .
\end{equation}
Figure~\ref{fig:colorill} illustrates how the sets $R_k (a,b)$ are intertwined.

\begin{figure}[ht]
\includegraphics[totalheight=.11in]{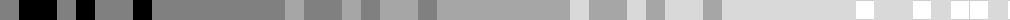}
\quad $\cdots$
\caption{The sets $R_0(3,5)$, $R_1(3,5)$, $R_2(3,5)$, etc.; the shading gets lighter as the index increases.}\label{fig:colorill}
\end{figure}

As an analogue to computing higher moments in statistics, it is natural to ask for higher power sums, or at least their nature. To this extent, we define
\[
  s_k^m \left( a_1, a_2, \dots, a_n \right) \ := \ \sum_{ j \in R_k \left( a_1, a_2, \dots, a_n \right) } j^m
\]
and offer Theorem~\ref{thm:powersums} below involving the \emph{Bernoulli polynomials} $B_n (x)$, defined as usual through
\[
  \frac{z \, e^{xz } }{e^{z } - 1 } = \sum_{n \geq 0 } \frac{B_n (x) }{ n! } \, z^n
\]
(see, e.g.,~\cite[Section~2.4]{ccd}).
The first few Bernoulli polynomials are
\begin{align*}
B_0 (x) \ &= \ 1 \, , \\
B_{1} (x) \ &= \ x - \tfrac{1}{2} \, , \\
B_{2} (x) \ &= \ x^{2} - x + \tfrac{1}{6} \, , \\
B_{3} (x) \ &= \ x^{3} - \tfrac{3}{2} x^{2} + \tfrac{1}{2} x \, , \\
B_{4} (x) \ &= \ x^{4} - 2 x^{3} + x^{2} - \tfrac{1}{30} \, , \\
B_{5} (x) \ &= \ x^{5} - \tfrac{5}{2} x^{4} + \tfrac{5}{3} x^{3} - \tfrac{1}{6} x \, , \\
B_{6} (x) \ &= \ x^{6} - 3 x^{5} + \tfrac{5}{2} x^{4} - \tfrac{1}{2} x^{2} + \tfrac{1}{42} \, .
\end{align*}
The crucial property of Bernoulli polynomials that we will need is (see, e.g.,~\cite[Lemma~2.3]{ccd})
\begin{equation}\label{eq:betacrucial}
  \beta_k(x) \ := \ \frac 1 k \left( B_k(x) - B_k(0) \right) \ = \ \sum_{ j=0 }^{ x-1 } j^{ k-1 } .
\end{equation}

\begin{theorem}\label{thm:powersums}
Let $a$ and $b$ be relatively prime positive integers, $k \ge 1$, and $m \ge 0$.
Then
\[
  s_k^m(a,b) \ = \ \sum_{ \lambda + \mu + \nu = m } \binom{ m }{ \lambda \ \mu \ \nu } a^{ \lambda + \mu } b^{ \lambda + \nu } (k-1)^\lambda
\beta_{ \nu + 1 } (a) \, \beta_{ \mu + 1 } (b) \, .
\]
\end{theorem}

This generalizes the above results for $c_k(a,b)$ (which is the case $m=0$) and $s_k(a,b)$ (the case $m=1$), and it gives the asymptotic statement that $s_k^m(a,b)$ is a
polynomial in $k$ of degree $m$ with leading coefficient $(ab)^{ m+1 }$. 

There are other concepts and results hidden in our generating functions.
To give a taste, we recall that $S_k \left( a_1, a_2, \dots, a_n \right)$ consists of all integers with more than $k$ representations in the
form~\eqref{repequ}, for general $n$. Thus $\ZZ_{ \ge 0 } \setminus S_k \left( a_1, a_2, \dots, a_n \right)$ consists of all nonnegative integers with
at most $k$ representations. We define
\begin{itemize}
  \item $g_{ \le k } \left( a_1, a_2, \dots, a_n \right)$ as the maximal integer in $\ZZ_{ \ge 0 } \setminus S_k \left( a_1, a_2, \dots, a_n \right)$;
  \item $c_{ \le k } \left( a_1, a_2, \dots, a_n \right)$ as the cardinality of $\ZZ_{ \ge 0 } \setminus S_k \left( a_1, a_2, \dots, a_n \right)$;
  \item $s_{ \le k } \left( a_1, a_2, \dots, a_n \right)$ as the sum of all elements in $\ZZ_{ \ge 0 } \setminus S_k \left( a_1, a_2, \dots, a_n \right)$.
\end{itemize}
In words, $g_{ \le k } \left( a_1, a_2, \dots, a_n \right)$ is the largest integer with at most $k$ representations, $c_{ \le k } \left( a_1, a_2, \dots, a_n \right)$ is the number of integers with at most $k$ representations, and $s_{ \le k } \left( a_1, a_2, \dots, a_n \right)$ is the sum of all integers with at most $k$ representations.

The following result can be proved directly from the first part of Theorem~\ref{thm:kgenfct}.
(We note that the formulas for $g_{ \le k }(a,b)$ and $c_{ \le k }(a,b)$ are not new.)

\begin{corollary}
Let $a$ and $b$ be relatively prime positive integers and $k \ge 0$. Then
\begin{itemize}
  \item $g_{ \le k } (a, b) = (k+1)ab - a - b$;
\vspace{.04in}
  \item $c_{ \le k } (a, b) = \tfrac 1 2 (a-1)(b-1) + abk$;
\vspace{.04in}
  \item $s_{ \le k } (a, b) = \tfrac{ 1 }{ 2 } \, a^2 b^2 k^2 + \tfrac{ 1 }{ 2 } \left( ab - a - b \right) abk + \tfrac 1 6 \, a^2 b^2 
  - \tfrac 1 4 \left( a+b-1 \right) ab + \tfrac{ 1 }{ 12 } \left( a^2 + b^2 - 1 \right) . $
\end{itemize}
\end{corollary}

We remark that $g_{ \le k } \left( a_1, a_2, \dots, a_n \right) = g_k \left( a_1, a_2, \dots, a_n \right)$ holds only for $n=2$; in fact, for general~$n$ these two invariants can differ quite a bit~\cite{frobextreme,shallitstankewiczfrobenius}.


\section{Proofs}

\begin{proof}[Proof of Theorem~\ref{thm:Sgenfct}]
Let
\[
  r(a, b; j) \ := \ \left| \left\{ (m,n) \in \ZZ_{ \ge 0 }^2 : \, ma+nb = j \right\} \right| ,
\]
the number of representations of $j$ in terms of $a$ and $b$.
By a simple geometric series argument,
\begin{equation}\label{eq:geomseries}
  \sum_{ j \ge 0 } r(a, b; j) \, z^j \ = \ \frac{ 1 }{ (1-z^a) (1-z^b) } \, .
\end{equation}

We claim that
\begin{equation}\label{eq:combex}
  r(a, b; j) \le 1 \ \text{ for } j < ab
\qquad \text{ and } \qquad
  r(a, b; j) = r(a, b; j-ab) + 1 \ \text{ for } j \ge ab \, ,
\end{equation}
and so, in particular, any integer $\ge ab$ belongs to $S_0 (a,b)$.
There are several ways to prove~\eqref{eq:combex}, for example, by considering the set
\[
  M_j \ := \ \left\{ m \in \ZZ_{ \ge 0 } : \, ma + nb = j \text{ for some } n \in \ZZ_{ \ge 0 } \right\} .
\]
Then $M_j \subset [0, \frac j a]$ (because $m, n \in \ZZ_{ \ge 0 }$), and indeed, if $m \in M_j$, then $M_j = [0, \frac j a]
\cap (m+b\ZZ)$, which follows from basic number theory (and here the condition $\gcd(a,b)=1$ is vitally important).
Thus, for $j < ab$, the set $M_j$ contains at most one element.
For $j \ge ab$, we have the implication $m \in M_{ j-ab } \Longrightarrow m \in M_j$, by replacing $n$ with $n+a$. Moreover,
the set difference $M_j \setminus M_{ j-ab } = ( \frac j a - b, \frac j a ] \cap (m+b\ZZ)$ contains precisely one point, and
\eqref{eq:combex} follows.

By~\eqref{eq:combex},
\[
  \sum_{ j \in S } z^j
  \ = \ \sum_{ j=0 }^{ ab-1 } r(a, b; j) \, z^j + \sum_{ j \ge ab } \left( r(a, b; j) - r(a, b; j-ab) \right)
z^j
  \ = \ \left( 1 - z^{ ab } \right) \sum_{ j \ge 0 } r(a, b; j) \, z^j .
\]
Theorem~\ref{thm:Sgenfct} follows now with~\eqref{eq:geomseries}.
\end{proof}

\begin{proof}[Proof of Theorem~\ref{thm:kgenfct}]
We proceed by induction on $k$; the base case is Theorem~\ref{thm:Sgenfct}.
For the induction step, assume that 
\[
 \sum_{ j \in S_{ k-1 } (a,b) } z^{j} \ = \ \frac{ z^{ ab(k-1) } (1 - z^{ ab }) }{ (1-z^a) (1-z^b) } \, .
\]
Now~\eqref{eq:combex} implies for $j \ge ab$ and $k \ge 1$
\[
  j \in S_k (a,b) \qquad \Longleftrightarrow \qquad j-ab \in S_{ k-1 } (a,b)
\]
(we stress once more that this heavily depends on $a$ and $b$ being relatively prime),
and so by induction hypothesis,
\begin{align*}
 \sum_{ j \in S_k (a,b) } z^j
 \ &= \ \sum_{ j \in S_{ k-1 } (a,b) } z^{j+ab} \\
 &= \ z^{ ab } \, \frac{ z^{ ab(k-1) } (1 - z^{ ab }) }{ (1-z^a) (1-z^b) } \\
 &=\frac{ z^{ abk } (1 - z^{ ab }) }{ (1-z^a) (1-z^b) } \, .
 \end{align*}
The formula for $p_k \left( a, b; z \right)$ now follows from the fact that $R_k (a,b) = S_{ k-1 } (a,b) \setminus S_k(a,b)$.
\end{proof}

\begin{proof}[Proof of Theorem~\ref{thm:powersums}]
We start by noting that the operator $\Delta := z \, \frac{ d }{ dz }$ is very useful in studying our power sums, as
\[
  \Delta \, z^j \ = \ z \, \frac{ d }{ dz } \, z^j \ = \ j \, z^j
\]
and thus 
\[
  s_k^m(a,b) \ = \ \lim_{ z \to 1 } \Delta^m p_k \left( a, b; z \right) \, .
\]
The operator $\Delta$ satisfies the same product rule as the derivative, and so by~\eqref{eq:polynomialversion},
\[
  \Delta^m \left( p_k \left( a, b; z \right) \right)
  \ = \ \sum_{ \lambda + \mu + \nu = m } \binom{ m }{ \lambda \ \mu \ \nu } \Delta^\lambda \left( z^{ ab(k-1) } \right) \Delta^\mu \left( \sum_{ j=0 }^{ b-1 } z^{ ja } \right) \Delta^\nu \left( \sum_{ j=0 }^{ a-1 } z^{ jb } \right) 
\]
and thus
\begin{align*}
  s_k^m(a,b) \ &= \ \sum_{ \lambda + \mu + \nu = m } \binom{ m }{ \lambda \ \mu \ \nu } \left( ab(k-1) \right)^\lambda \left( \sum_{ j=0 }^{ b-1 } (ja)^\mu \right) \left( \sum_{ j=0 }^{ a-1 } (jb)^\nu \right) \\
          &= \ \sum_{ \lambda + \mu + \nu = m } \binom{ m }{ \lambda \ \mu \ \nu } a^{ \lambda + \mu } b^{ \lambda + \nu } (k-1)^\lambda \left( \sum_{ j=0 }^{ b-1 } j^\mu
\right) \left( \sum_{ j=0 }^{ a-1 } j^\nu \right) .
\end{align*}
We finish by substituting for the expressions in the last two parentheses using~\eqref{eq:betacrucial}.
\end{proof}


\section{Musings about $n \ge 3$}

The reader might have noticed the striking similarities between the rational generating function in Theorem~\ref{thm:Sgenfct} and that in~\eqref{eq:geomseries}; however, this is an artifact of the case $n=2$. While it is true that the general counting function
\[
  r(a_1, a_2, \dots, a_n; j) \ := \ \left| \left\{ \bm \in \ZZ_{ \ge 0 }^n : \, m_1 a_1 + m_2 a_2 + \dots + m_n a_n = j \right\} \right|
\]
comes with the generating function
\[
  \sum_{ j \ge 0 } r(a_1, a_2, \dots, a_n; j) \, z^j \ = \ \frac{ 1 }{ (1-z^{ a_1 } ) (1-z^{ a_2 } ) \cdots (1-z^{ a_n } ) } \, ,
\]
and also that 
\begin{equation}\label{eq:gengenfct}
  \sum_{ j \in S_0 \left( a_1, a_2, \dots, a_n \right) } z^j \ = \ \frac{ h(z) }{ (1-z^{ a_1 } ) (1-z^{ a_2 } ) \cdots (1-z^{ a_n } ) }
\end{equation}
for some polynomial $h(z)$, the form of $h(z)$ is simple only for $n \le 2$.
At any rate, Denham \cite{denham} discovered the remarkable fact that for $n=3$, the polynomial $h(z)$
has either $4$ or $6$ terms. He gave semi-explicit formulas for $h(z)$, from which one can
deduce a semi-explicit formula for the Frobenius number $g_0(a_1, a_2, a_3)$. This formula was
independently found by Ram\'irez-Alfons\'in \cite{ramirezhilbert}. Denham's theorem implies
that the Frobenius number in the case $n=3$ is quickly computable, which was previously
known \cite{davison,greenberg,herzog}. 
Bresinsky \cite{bresinsky} proved that for $d \ge 4$, there is no absolute bound for the
number of terms in $h(z)$, in sharp contrast to Denham's theorem.

On the computational side, Barvinok--Woods \cite{barvinokwoods} proved that for fixed $n$, the
rational generating function \eqref{eq:gengenfct} can be written as a short sum of rational functions; in
particular, \eqref{eq:gengenfct} can be efficiently computed when $n$ is fixed. A
corollary of this fact is that the Frobenius number can be efficiently computed when $n$ is
fixed, a theorem originally due to Kannan \cite{kannan}. 
The analogous result for the generalized Frobenius numbers $g_k \left( a_1, a_2, \dots, a_n \right)$ is due to
Aliev--De Loera--Louveaux~\cite{alievdeloeralouveaux}.
On the other hand, Ram\'irez-Alfons\'in \cite{ramirezfrobcomplexity} proved that trying to efficiently compute the Frobenius number is hopeless if $n$ is left as a variable.

As a final note, while our results give a clear picture what kind of functions to expect for $n=2$---e.g., $s_k(a,b)$ is linear in $k$ and $s_{ \le k } (a,b)$ is quadratic in
$k$---it is unclear to us how this generalizes to $n \ge 3$. Some basic structural results would undoubtedly shed new light on generalized Frobenius numbers and their
relatives.


\bibliographystyle{amsplain}

\begin{thebibliography}{10}

\bibitem{alievdeloeralouveaux}
Aliev, I., De~Loera, J.\ A., Louveaux, Q. (2016). Parametric polyhedra with at least {$k$} lattice points: their
semigroup structure and the {$k$}-{F}robenius problem.  
Beveridge, A., Griggs, J. R., Hogben, L., Musiker, G., Tetal, P., eds.
\emph{Recent Trends in Combinatorics}. IMA Volumes in Mathematics 159. Cham: Springer, pp.~753--778.


\bibitem{barvinokwoods}
Barvinok, A., Woods, K. (2003). Short rational generating functions for lattice point problems. \emph{J. Amer. Math.  Soc.} 16(4): 957--979.

\bibitem{frobextreme}
Beck, M., Kifer, C. (2011). An extreme family of generalized {F}robenius numbers. \emph{Integers}. 11(A24): 6~pp.

\bibitem{frobnote}
Beck, M., Robins, S. (2004). A formula related to the {F}robenius problem in two dimensions. 
Chudnovsky, D., Chudnovsky, G., Nathanson, M., eds.
\emph{Number theory (New York, 2003)}. New York: Springer, pp.~17--23.

\bibitem{ccd}
Beck, M., Robins, S. (2015). \emph{Computing the Continuous Discretely: Integer-point Enumeration in Polyhedra}. 2nd ed. New York: Springer.

\bibitem{bresinsky}
Bresinsky, H. (1975). Symmetric semigroups of integers generated by {$4$} elements. \emph{Manuscripta Math.} 17(3): 205--219.

\bibitem{brownshiue}
Brown, T.\ C., Shiue, P. J.-S. (1993). A remark related to the {F}robenius problem. \emph{Fibonacci Quart.} 31(1): 32--36.

\bibitem{carlitzmonthly}
Carlitz, L. (1966). The number of terms in the cyclotomic polynomial {$F_{pq}(x)$}. \emph{Amer. Math. Monthly}. 73: 979--981.

\bibitem{davison}
Davison, J.\ L. (1994). On the linear {D}iophantine problem of {F}robenius. \emph{J. Number Theory}. 48(3): 353--363.

\bibitem{denham}
Denham, G. (2003). Short generating functions for some semigroup algebras. \emph{Electron. J. Combin.} 10(36): 7~pp.

\bibitem{greenberg}
Greenberg, H. (1980). An algorithm for a linear {D}iophantine equation and a problem of {F}robenius. \emph{Numer. Math.} 34(4): 349--352.

\bibitem{herzog}
Herzog, J. (1970). Generators and relations of abelian semigroups and semigroup rings. \emph{Manuscripta Math.} 3: 175--193.

\bibitem{kannan}
Kannan, R. (1992). Lattice translates of a polytope and the {F}robenius problem. \emph{Combinatorica}. 12(2): 161--177.

\bibitem{moreemonthly}
Moree, P. (2014). Numerical semigroups, cyclotomic polynomials, and {B}ernoulli numbers. \emph{Amer. Math. Monthly}. 121(10): 890--902.

\bibitem{nijenhuis}
Nijenhuis, A. (1979). A minimal-path algorithm for the ``money changing problem.'' \emph{Amer. Math. Monthly}. 86(10): 832--835.


\bibitem{ramirezfrobcomplexity}
Ram{\'{\i}}rez--Alfons{\'{\i}}n, J.\ L. (1996). Complexity of the {F}robenius problem. \emph{Combinatorica}. 16(1): 143--147.

\bibitem{ramirezhilbert}
Ram{\'{\i}}rez--Alfons{\'{\i}}n, J.\ L. (2002). The {F}robenius number via {H}ilbert series. Preprint.

\bibitem{ramirezbook}
Ram{\'{\i}}rez--Alfons{\'{\i}}n, J.\ L. (2005). \emph{The {D}iophantine {F}robenius {P}roblem}. Oxford Lecture
Series in Mathematics and its Applications, vol.~30. Oxford: Oxford Univ. Press. 

\bibitem{sertozozluk2}
Sert{\"o}z, S., {\"O}zl{\"u}k, A.\ E. (1991). On the number of representations of an integer by a linear form.
\emph{{I}stanbul \"{U}niv. Fen Fak. Mat. Derg.} 50: 67--77.

\bibitem{shallitstankewiczfrobenius}
Shallit, J., Stankewicz, J. (2011). Unbounded discrepancy in {F}robenius numbers. \emph{Integers}. 11(A2): 8~pp.

\bibitem{sylvester}
Sylvester, J.\ J. (1884). Mathematical questions with their solutions. \emph{Educational Times}. 41: 171--178.

\bibitem{szekelywormald}
Sz\'{e}kely, L.\ A., Wormald, N.\ C. (1986). Generating functions for the {F}robenius problem with {$2$} and {$3$} generators. \emph{Math.\ Chronicle}. 15: 49--57.

\bibitem{wilfmonthly}
Wilf, H.\ S. (1978). A circle-of-lights algorithm for the ``money-changing problem.'' \emph{Amer. Math. Monthly}. 85(7): 562--565.

\end{thebibliography}

\setlength{\parskip}{0cm} 

\end{document}